\documentclass{amsart}

\usepackage[english]{babel}
\usepackage{array,booktabs,float}
\usepackage{amsmath,amssymb,amsfonts}
\usepackage{amsthm}
\usepackage{bm}
\usepackage{amscd}
\usepackage{mathtools}
\usepackage{mathrsfs}
\usepackage{pb-diagram}
\usepackage{booktabs}
\usepackage{cleveref}
 
\usepackage{todonotes}
\usepackage{comment}
\usepackage{tikz-cd}
\usepackage{url}

\theoremstyle{plain}
\newtheorem{theorem}{Theorem}[section]
\crefname{theorem}{Theorem}{Theorems}
\Crefname{theorem}{Theorem}{Theorems}
\newtheorem{proposition}[theorem]{Proposition}
\crefname{proposition}{Proposition}{Propositions}
\Crefname{proposition}{Proposition}{Propositions}
\newtheorem{lemma}[theorem]{Lemma}
\crefname{lemma}{Lemma}{Lemmas}
\Crefname{lemma}{Lemma}{Lemmas}
\newtheorem{corollary}[theorem]{Corollary}
\crefname{corollary}{Corollary}{Corollaries}
\Crefname{corollary}{Corollary}{Corollaries}

\crefname{claim}{Claim}{Claims}
\Crefname{claim}{Claim}{Claims}

\crefname{property}{Property}{Properties}
\Crefname{property}{Property}{Properties}

\crefname{problem}{Problem}{Problems}
\Crefname{problem}{Problem}{Problems}

\theoremstyle{definition}
\newtheorem{definition}[theorem]{Definition}
\crefname{definition}{Definition}{Definitions}
\Crefname{definition}{Definition}{Definitions}

\crefname{notation}{Notation}{Notations}
\Crefname{notation}{Notation}{Notations}

\crefname{convention}{Convention}{Conventions}
\Crefname{convention}{Convention}{Conventions}

\crefname{condition}{Condition}{Conditions}
\Crefname{condition}{Condition}{Conditions}

\crefname{assumption}{Assumption}{Assumptions}
\Crefname{assumption}{Assumption}{Assumptions}

\theoremstyle{remark}

\crefname{remark}{Remark}{Remarks}
\Crefname{remark}{Remark}{Remarks}

\crefname{example}{Example}{Examples}
\Crefname{example}{Example}{Examples}

\crefname{section}{Section}{Sections}
\Crefname{section}{Section}{Sections}
\crefname{subsection}{Subsection}{Subsections}
\Crefname{subsection}{Subsection}{Subsections}
\crefname{figure}{Figure}{Figures}
\Crefname{figure}{Figure}{Figures}

\newtheorem*{acknowledgement}{Acknowledgement}

\newcommand{\Z}{\mathbb{Z}}
\newcommand{\Q}{\mathbb{Q}}
\newcommand{\R}{\mathbb{R}}
\newcommand{\C}{\mathbb{C}}
\newcommand{\q}{\mathbb{H}}
\newcommand{\ind}{\mathrm{ind}}
\newcommand{\Slash}[1]{{\ooalign{\hfil/\hfil\crcr$#1$}}}

\begin{document}

\title{Boundary Dehn twists on Milnor fibers and Family Bauer--Furuta invariants}
\author{Jin Miyazawa}
\address{Reserch Institute for Mathematical Sciences, Kyoto University, Kyoto 606-8502, Japan}
\email{miyazawa.jin.5a@kyoto-u.ac.jp}
\begin{abstract}
    We proved that the boundary Dehn twist on the Milnor fiber $M_c(2, q, r)$ is an exotic diffeomorphism relative to the boundary if $q, r$ are odd, coprime integers bigger than $3$ and $(q-1)(r-1)/4$ is an odd number. The proof is given by comparing the family relative Bauer--Furuta invariants of the mapping torus. 
\end{abstract}
\maketitle
\section{Introduction}
Let $X$ be a $4$-manifold and $f$ be a diffeomorphism $f\colon X \to X$.If $f$ represents a non-trivial element of the kernel of the map
\[
\pi_0(\text{Diff}(X)) \to \pi_0(\text{Homeo}(X)), 
\]
then $f$ is said to be an exotic diffeomorphism. Moreover, if $\partial X \neq \emptyset$ and a diffeomorphism $f$ which is the identity on the boundary represents a non-trivial element of the kernel of the map
\[
    \pi_0(\text{Diff}(X, \partial X)) \to \pi_0(\text{Homeo}(X, \partial X)), 
\]
then $f$ is said to be an exotic diffeomorphism relative to the boundary.  The first example of an exotic diffeomorphism is given by Ruberman \cite{ruberman1998obstruction}. 

Recently, some examples of exotic diffeomorphisms are given by Dehn twist along three manifolds embedded in $4$-manifolds. Some examples of exotic diffeomorphisms relative to the boundary are given by boundary Dehn twists (see \cref{examples}). 
Kronheimer--Mrowka \cite{kronheimer2020dehn} proved that the Dehn twist along $S^3$ in the neck of the connected sum $K3\# K3$ is an exotic diffeomorphism and 
Lin \cite{lin2023isotopy} proved that that Dehn twist remains exotic after taking connected sum with $S^2 \times S^2$. 
Baraglia--Konno \cite{baraglia2022bauer} proved that the boundary Dehn twist on $K3 \setminus B^4$ is an exotic diffeomorphism relative to the boundary. Qiu \cite{qiu2024dehn} proved that the Dehn twist along $S^3$ in the neck of the connected sum $X_1 \# X_2$ where $X_1$ and $X_2$ are homology $4$-torus is exotic. 

The above examples are studies of the Dehn twists along  $S^3$ in $4$-manifolds. Other interesting examples, which are main topics of this paper, are the boundary Dehn twist along Seifert-fibered $3$-manifolds. 
These Dehn twist are also studied recently. For example, 
Konno--Mallick--Taniguchi \cite{konno2023exotic},  
Konno--Lin--Mukherjee--Mu{\~n}oz-Ech{\'a}niz \cite{konno2024four},
and Kang--Park--Taniguchi \cite{kang2024exotic} proved exoticness of such kinds of diffeomorphisms relative to the boundary under some conditions. 

In this paper, we study boundary Dehn twists on the Milnor fiber. 
Let us state the main applications of this paper.

\begin{theorem}\label{milnor_fiber}
    Let $M_c(p, q, r)$ be the Milnor fiber that is given by
    \[
        M_c(p, q, r)=\{(x, y, z) \in \C^3\mid x^{p}+y^{q}+z^{r}=\epsilon, \; \lvert x \rvert^2+\lvert y \rvert^2+\lvert z \rvert^2 \le 1 \}
    \]    
    for some small $\epsilon \in \C \setminus 0$. 
    If $p=2$ and $q, r$ are odd, coprime, and 
    \[
    \frac{(q-1)(r-1)}{4}
    \] is an odd number, then the boundary Dehn twist on $M_c(2, q, r)$ is an exotic diffeomorphism relative to the boundary. 
\end{theorem}
In the case that $(p, q, r)=(2, 3, 7)$ or $(2, 3, 11)$, the exoticness of the boundary Dehn twists were proved in \cite{konno2023exotic}. 
In the case that the boundary is Floer simple, then the exoticness was proved in \cite{konno2024four}. 
\begin{theorem}\label{milnor2}
    Let us assume that $(p, q, r)$ are pairwise coprime and $2 \le p < q < r$ and $7 \le r$. 
    Then there is an exotic diffeomorphism relative to the boundary on $M_c(p, q, r)$. 
\end{theorem} 

The above applications are proved using the following theorem. 

\begin{theorem}\label{main}
Let $X$ be a compact connected smooth spin $4$-manifold with boundary $Y$. We assume that $H_1(X, \Q)=0$ and $Y$ is a rational homology sphere. Let $f$ be a self-diffeomorphism of $X$ such that $f|_Y=id_Y$. We assume the following conditions:
\begin{itemize}
    \item $f^*=id$ on $H^*(X, \Z)$. 
    \item There is a self-diffeomorphism $\delta$ such that $\delta^2=f$ and $\delta|_Y=id_Y$.
    \item The dimension of the $-1$-eigenspace of $\delta^*$ in $H^+(X, \R)$ is even and not divisible by $4$.  
\end{itemize}
If the one of the following conditions is satisfied, then $f$ is not smoothly isotopic to the identity relative to the boundary. 
\begin{enumerate}
    \item \label{closure} There is a compact oriented smooth $4$-manifold $X^+$ such that $\partial X^+=-Y$ and $b^+(X \cup_Y X^+)>1$. We also assume that there is a spin$^c$ structure $\mathfrak{s}$ on $X \cup_Y X^+$ such that $\mathfrak{s}|_X$ is a spin structure and the Seiberg--Witten invariant $SW(\mathfrak{s})$ is odd. 
    \item \label{symplectic} $X$ is symplectic and the canonical spin$^c$ structure induced by the symplectic structure coincides with the spin$^c$ structure induced by a spin structure.  
\end{enumerate}
\end{theorem}

If $X$ is simply connected, then the topological triviality of the map $f$ follows from the theorem proved by Orson--Powell \cite[Theorem A]{orson2022mapping}. 

To prove \cref{main}, we compare family relative Bauer--Furuta invariants of mapping torus of $f$ and the trivial family by taking a pairing with a family relative Bauer--Furuta invariant of a trivial family of $4$-manifolds with boundary $-Y$ (case \eqref{closure} in \cref{main}) or a Floer homotopy contact invariant of $-Y$ (case \eqref{symplectic} in \cref{main}). 
In \cite{konno2024four} and \cite{iida2024diffeomorphisms}, they consider the pairing of relative family Seiberg--Witten invariant with the monopole Floer--homology valued invariant of contact structures. 
In their case, the formal dimension of the Seiberg--Witten moduli space for each fiber is $-1$ while we consider the case in which the formal dimension for each fiber is $0$. 

We consider a numerical invariant defined by using the family Bauer--Furuta invariant. To compute that invariant, we use a square root of the diffeomorphism with specific properties that appear in \cite{konno2023exotic}; however the method of using the square root differs from the original one. 

\begin{acknowledgement}
    The author is grateful to Tasuki Kinjo, Hokuto Konno, and Masaki Taniguchi for helpful discussions. 
    The author is supported by JSPS KAKENHI Grant Number 24027367. 
\end{acknowledgement}

\section{Review of the family Bauer--Furuta invariant}
    In this section, we recall the definition of the family Bauer--Furuta invariant that is given in \cite[Subsection 2.3]{konno2022groups}. Let us consider the following setup and introduce some notations. 
    \begin{itemize}
        \item Let $X$ be an oriented connected compact $4$-manifold with boundary. We assume that the boundary of $X$ is a rational homology sphere. 
        \item Let $\pi \colon \mathbb{X} \to B$ be an oriented smooth  fiber bundle with fiber $X$ on a compact topological space $B$. 
        \item Assume that there is a trivialization  on $\partial \mathbb{X}$, which is the family of the boundary of $X$ on $B$. We fix a trivialization $\phi \colon \partial \mathbb{X} \to B \times \partial X$. 
        \item Assume that there is a a fiber-wise spin structure $\boldsymbol{\mathfrak{s}}$ on $\mathbb{X}$. We also assume there is a trivialization $\boldsymbol{\mathfrak{s}}|_{\partial \mathbb{X}} \to B \times \mathfrak{t}$ for some spin structure $\mathfrak{t}$ on $\partial X$. We fix this trivialization. 
        \item We write $H^+(\mathbb{X})$ the vector bundle on $B$ with its fiber on $b \in B$ is $H^+(\mathbb{X}_b, \R) \cong H^+(X, \R)$.
        \item Let $g$ be a Riemannian metric on $Y$ and $\tilde{g}$ be a fiberwise Riemannian metric on $\mathbb{X}$ such that $\tilde{g}$ is cylindrical near the boundary $\partial \mathbb{X}$.    
    \end{itemize}
    Then we can define the family Bauer--Furuta invariant 
    \[
    \Psi(\mathbb{X}, \boldsymbol{\mathfrak{s}}) \in \{  Th(\ind_{\C}(\Slash{D}_g)), Th(H^+(\mathbb{X}))\wedge \Sigma^{(n(Y, g)/2)\q}SWF(Y)\}_{S^1}
    \] 
    where $\ind_{\C}(\Slash{D}_{\tilde g}) \in K(B)$ is the family index of the spin structure $\mathfrak{s}$ and $n(Y, g)$ is the correction term defined by Manolescu \cite{manolescu2003seiberg}. 
$Th(V)$ is the Thom space of a vector bundle $V$. 

One can prove that the $S^1$-equivariant stable homotopy class of $\Psi(\mathbb{X}, \boldsymbol{\mathfrak{s}})$ is independent of the choice of a family of metrics and reference connections by a standard argument.  
    
    \subsection{Family on $S^1$}
        Let us consider the family Bauer--Furuta invariants in the case that $B=S^1$. 
    In this case, since $\mathfrak{s}$ is a spin structure, we have a canonical trivialization of $\ind_{\C}(\Slash{D})$ using the quaternionic structure on $\ind_{\C}(\Slash{D})$, see the paper of Kronheimer--Mrowka \cite{kronheimer2020dehn}. If a trivialization $o\colon H^+(\mathbb{X}) \to \R^{b^+(X)} \times S^1$ is given, by taking a representative of $\Psi(\mathbb{X}, \boldsymbol{\mathfrak{s}})$ we have the following maps:
    \[
        S^{(n\R \oplus ({-\frac{\sigma(X)}{16}+m})\q)}(S^1) \to  S^{({n+b^+(X)}) \R \oplus m \q}(S^1)\wedge SWF(Y) 
    \]
for $n, m \gg 1$ where $S^{V}(A)$ denote the unreduced suspension of $V$ on $A$. If we fix an isomorphism $\q \to \R^4$ and forget $\mathrm{Pin}(2)$-action, we have a stable homotopy class
\begin{equation}\label{family_invariant}
    \Phi'(\mathbb{X}, {\mathfrak{s}}, o) \in \{ S^{-\frac{\sigma(X)}{4} -b^+(X)}(S^1), S^1_+ \wedge SWF(Y) \}  
\end{equation}
where $S^1_+$ is the disjoint union of the $S^1$ and a point. 
It is difficult to compute the invariant $\Phi'(\mathbb{X}, \mathfrak{s}, o)$ in \eqref{family_invariant} when $SWF(Y)$ is not the sphere spectrum. In the next section, we define two types of  $\Z/2\Z$-valued invariants from $\Phi'(\mathbb{X}, \mathfrak{s}, o)$. Each invariant corresponding to the conditions \eqref{closure} or \eqref{symplectic} in \cref{main}.  

\section{Definitions of our invariants and calculations}
In this section we give two types of $\Z/2\Z$-valued invariants corresponding to the conditions \eqref{closure} or \eqref{symplectic} in \cref{main}. 
\subsection{Definition of the $\Z/2\Z$-invariant: case \eqref{closure}}\label{sec_case1} 
    Let $X$ be a $4$-manifold and $\varphi$ be a diffeomorphism of $X$ that satisfies the following conditions:
    \begin{itemize}
        \item $X$ is a compact connected spin $4$-manifold and $H_1(X, \Q)=0$.  Moreover its boundary $Y$ is a rational homology sphere. 
        \item The diffeomorphism $\varphi$ is the identity on the boundary. 
        \item The induced map $\varphi^* \colon \det H^+(X) \to \det H^+(X)$ is the identity. 
        \item There is a compact $4$-manifold $X^+$ with $\partial X^+ =-Y$, $b^1(X^+)=0$, and $b^+(X \cup_Y X^+) >1$. 
        \item Let $\mathfrak{s}$ be the unique spin structure on $X$. Then there is a spin$^c$ structure $\mathfrak{s}'$ on $X \cup_Y X^+$ such that $\mathfrak{s}'|_X \cong \mathfrak{s}$ as a spin$^c$ structure. 
        \item The formal dimension of the Seiberg--Witten moduli space of the spin$^c$ structure $\mathfrak{s}'$ on $X \cup_Y X^+$ is $0$. 
\end{itemize}
We write $X'$ for the glued $4$-manifold $X \cup_Y X^+$. 
Let $\mathbb{X}$ be the mapping torus of $\varphi$ and $\pi \colon \mathbb{X} \to S^1$ be the projection. 
Since the map $\varphi^* \colon \det H^+(X) \to \det H^+(X)$ is the identity, we have a framing of the bundle $H^+(\mathbb{X})$. Let us fix a framing of $H^+(\mathbb{X})$, say, $f$. 
Then we have an invariant 
    \[
    \Phi'(\mathbb{X}, \mathfrak{s}, f) \in \{ S^{-\frac{\sigma(X)}{4} -b^+(X)}(S^1), S^1_+ \wedge SWF(Y) \}
    \]
    as in the previous section. 

Let $\Phi(X^+, \mathfrak{s}'|_{X^+})$ be the unparametrized relative Bauer--Furuta invariant 
    \[
        \Phi(X^+, \mathfrak{s}'|_{X^+}) \in \{ S^{-\frac{\sigma(X^+)}{4}-b^+(X^+)}, SWF(-Y) \}. 
    \]
\begin{definition}\label{case1}
    Let $\epsilon$ be the Spaniel--Whitehead duality map:
    \[
    \epsilon \colon SWF(Y) \wedge SWF(-Y) \to \mathbb{S}^0
    \]
    where $\mathbb{S}^0$ is the sphere spectrum. 
    Then we define the following stable homotopy element: 
    \begin{align}
        \Psi(\mathbb{X}, \mathfrak{s}, f) := id_{S^1_+}\wedge \epsilon(\Phi'(\mathbb{X}, \mathfrak{s}, f) \wedge \Phi(X^+, \mathfrak{s}'|_{X^+})) \in &\{ S^{-\frac{\sigma(X')}{4}-b^+(X')}(S^1), S^1_+ \}\\
&= \{ \Sigma(S^1_+), S^1_+ \}. 
    \end{align}
    
\end{definition}

Let us define a $\Z/2$-valued invariant of the isomorphism class of $(\mathbb{X}', \mathfrak{s}', f')$. 
Let $n \gg 1$ be a sufficiently large integer and let
    \[
    \Psi(\mathbb{X}, \mathfrak{s}, f)_n \colon 
    S^{n+1}(S^1) \to S^n(S^1)
    \]     
    be a representative of $\Psi(\mathbb{X}, \mathfrak{s}, f)$. 
    Let us denote by $\tau $ the generator of $\tilde{KO}^{n}(S^n(S^1))$. 
Then we have 
\begin{align*}
    \Psi(\mathbb{X}, \mathfrak{s}, f)_n^*\tau \in & \tilde{KO}^{n}(S^{n+1}(S^1))   \\ 
    & \cong KO^{-1}(S^1) \\
    & \cong \tilde{KO}^{-1}(S^1) \oplus KO^{-1}(pt) \\
    & \cong KO^{-2}(pt) \oplus KO^{-1}(pt) \\
    & \cong \Z/2\Z \oplus \Z/2\Z
\end{align*}
here $pt$ is a point in $S^1$.  
\begin{definition}
    Let us write 
\[\Psi(\mathbb{X}, \mathfrak{s}, f)_n^*\tau=(\psi_{-2}(\mathbb{X}', \mathfrak{s}', f), \psi_{-1}(\mathbb{X}', \mathfrak{s}', f)) \in KO^{-2}(pt) \oplus KO^{-1}(pt). 
\] 
One can see that $\psi_{-2}$ and $\psi_{-1}$ are independent of $n$ if $n \gg1 $. We define the family invariant $\Xi(\mathbb{X}', \mathfrak{s}', f) \in \Z/2\Z$ by $\psi_{-2}(\mathbb{X}', \mathfrak{s}', f)$. 
\end{definition}

From the gluing formula (\cite{manolescu2007gluing}, \cite{khandhawit2023unfolded}) of Bauer--Furuta invariants, we have the following identification of $\psi_{-1}(\mathbb{X}', \mathfrak{s}', f)$.
\begin{proposition}\label{psi1}
    Let $X'$ be the glued $4$-manifold $X \cup_Y X^+$ and let $\mathfrak{s}'$ be the spin$^c$ structure which we mentioned in the beginning of this subsection. If $b^+(X') \equiv 3 \mod 4$, $\psi_{-1}(\mathbb{X}', \mathfrak{s}', f) \in \Z/2\Z$ is the mod $2$ reduction of the Seiberg--Witten invariant of $(X', \mathfrak{s}')$:
    \[
    \psi_{-1}(\mathbb{X}', \mathfrak{s}', f) \equiv SW(X', \mathfrak{s}') \mod 2. 
    \]
\end{proposition}
\begin{proof}
    From the definition of $\Psi(\mathbb{X}, \mathfrak{s}, f)$, we have the following commutative diagram:
    \[\begin{CD}
    S^{n+1}(S^1) @>>{\Psi(\mathbb{X}, \mathfrak{s}, f) }> S^n(S^1) \\
    @AA{}A  @AA{}A \\
    S^{n+1}(pt) @>>{\Phi(X', \mathfrak{s}')}> S^n(pt)
    \end{CD}
    \]
    where $\Phi(X', \mathfrak{s}')$ is the Bauer--Furuta map of the pair $(X', \mathfrak{s}')$. From the definition of $\psi_{-1}(\mathbb{X}', \mathfrak{s}', f)$, we obtain the following formula by applying $\tilde{KO}^{n}$ to the above diagram:
    \begin{align*}
        \Phi(X', \mathfrak{s}')^* \tau_{n}= \psi_{-1}(\mathbb{X}', \mathfrak{s}', f)
    \end{align*}
    where $\tau_n$ is the Thom class of $\tilde{KO}^{n}(S^n(pt))$. From \cite[Proposition 4.4.]{bauer2004stable} and the assumptions that $b^+ \equiv 3 \mod 4$ and the formal dimension of the Seiberg--Witten moduli space of $(X', \mathfrak{s}')$ is zero \footnote{Actually this condition implies that $X'$ is an almost complex since we can take a non-vanishing section of $\Lambda^+$}, we have 
    \[
        \Phi(X', \mathfrak{s}')^* \tau_{n}= SW(X', \mathfrak{s}') \mod 2.
    \]
This finishes the proof. 
\end{proof}

\subsection{Definition of the $\Z/2\Z$-invariant: case \eqref{symplectic}}\label{sec_case2}

Let $X$ be a symplectic $4$-manifold with symplectic form $\omega$ and $\phi$ be a diffeomorphism of $X$ that satisfies the following conditions:
\begin{itemize}
    \item $X$ is compact connected spin $4$-manifold and $H_1(X, \Q)=0$. Moreover its boundary $Y$ is a rational homology sphere. 
    \item The canonical spin$^c$ structure $\mathfrak{s}_{\omega}$ on $X$ that is induced by the symplectic structure $\omega$ is isomorphic to the spin$^c$ structure $\mathfrak{s}$ induced by a spin structure. 
    \item The diffeomorphism $\varphi$ is the identity on the boundary. 
        
    \item The induced map $\varphi^* \colon \det H^+(X) \to \det H^+(X)$ is the identity. 
\end{itemize}
Let $\mathbb{X}$ be the mapping torus of $\varphi$ and $\pi \colon \mathbb{X} \to S^1$ be the projection. 
We fix a framing $f$ of the bundle $H^+(\mathbb{X})$ as in the case of the previous subsection. 
Then we have a family invariant 
    \[
    \Phi'(\mathbb{X}, \mathfrak{s}, f) \in \{ S^{-\frac{\sigma(X)}{4} -b^+(X)}(S^1), S^1_+ \wedge SWF(Y) \}. 
    \]

Since $X$ is a symplectic $4$-manifold, the boundary $Y$ of $X$ has a natural contact structure induced by $\omega$. Let us denote by $\xi$ the contact structure. 
In \cite{iida2021seiberg}, Iida and Taniguchi introduced the stable homotopy group of the Seiberg--Witten Floer homotopy type valued contact invariant:
\begin{align}
    \Phi(Y, \xi) \in \{ S^0, \Sigma^{\frac{1}{2}-d_3(-Y, [\xi])}SWF(-Y, \mathfrak{s}_{\xi})  \}
\end{align}
where $\mathfrak{s}_{\xi}$ is a spin$^c$ structure induced by the contact structure $\xi$ and $d_3(-Y, [\xi])$ be the homotopy invariant of $2$-plane field introduced Gompf \cite{gompf1998handlebody}. From the assumption, $\mathfrak{s}_{\xi}$ coincides with the spin structure. 

From the assumption on the diffeomorphism $\phi$, the boundary of the family $\partial \mathbb{X}$ is the trivial family: $\partial \mathbb{X} \cong S^1 \times Y$. Then one can define the following invariant: 

\begin{definition}\label{case2}
    Let $\epsilon$ be the Spaniel--Whitehead duality map:
    \[
    \epsilon \colon SWF(Y) \wedge SWF(-Y) \to \mathbb{S}^0 
    \]
    where $\mathbb{S}^0$ is the sphere spectrum. 
    Then we define the following stable homotopy element: 
    \begin{align*}
        \Psi_{\xi}(\mathbb{X}, \mathfrak{s}, f) := id_{S^1_+}\wedge \epsilon(\Phi'(\mathbb{X}, \mathfrak{s}, f) \wedge \Phi(Y, \xi)) \in &\{ S^{\langle e(S^+), [X, \partial X] \rangle}(S^1), S^1_+ \}\\
&= \{ S^1_+, S^1_+ \}. 
    \end{align*}
    where $S^+$ is the spinor bundle of $\mathfrak{s}_{\omega}$ on $X$. 
\end{definition}

Let us define a $\Z/2$-valued invariant of the isomorphism class of a family $(\mathbb{X}, \mathfrak{s}, f)$ using contact structure $\xi$ on $Y$. 
Let $n \gg 1$ be a sufficiently large integer and let
    \[
    \Psi_{\xi}(\mathbb{X}, \mathfrak{s}, f)_n \colon 
    S^{n}(S^1) \to S^n(S^1)
    \]     
    be a representative of $\Psi_{\xi}(\mathbb{X}, \mathfrak{s}, f)$. 
    Let us denote by $\tau $ the generator of $\tilde{KO}^{n}(S^n(S^1))$. 
Then we have 
\begin{align*}
    \Psi_{\xi}(\mathbb{X}, \mathfrak{s}, f)_n^*\tau \in & \tilde{KO}^{n}(S^{n}(S^1))   \\ 
    & \cong KO^{0}(S^1) \\
    & \cong \tilde{KO}^{0}(S^1) \oplus KO^{0}(pt) \\
    & \cong KO^{-1}(pt) \oplus KO^{0}(pt) \\
    & \cong \Z/2\Z \oplus \Z
\end{align*}
here $pt$ is a point in $S^1$.  
\begin{definition}
    Let us write 
\[\Psi_{\xi}(\mathbb{X}, \mathfrak{s}, f)_n^*\tau=(\psi_{\xi, -1}(\mathbb{X}, \mathfrak{s}, f), \psi_{\xi, 0}(\mathbb{X}, \mathfrak{s}, f)) \in KO^{-1}(pt) \oplus KO^{0}(pt). 
\] 
One can see that $\psi_{\xi, -1}$ and $\psi_{0}$ are independent of $n$ if $n \gg1 $. We define the family invariant $\Xi_{\xi}(\mathbb{X}, \mathfrak{s}, f) \in \Z/2\Z$ by $\psi_{\xi, -1}(\mathbb{X}, \mathfrak{s}, f)$. 
\end{definition}

Let $\Phi(X, \xi, \mathfrak{s})$ be a Bauer--Furuta refinement of the Kronheimer--Mrowka's invariant introduced by Iida in \cite{iida2021bauer}. Then we can calculate $\psi_{\xi, 0}(\mathbb{X}, \mathfrak{s}, f)$. 

\begin{proposition}\label{psi1}
    Under the assumptions on $X$, $\xi$, and $\mathfrak{s}$ which we described in the beginning of this subsection,  $\psi_{\xi, 0}(\mathbb{X}, \mathfrak{s}, f)$ is a generator of $KO^0(pt) \cong \Z$. 
\end{proposition}
\begin{proof}
    From the definition of $\Psi(\mathbb{X}, \mathfrak{s}, f)$, we have the following commutative diagram:
    \[\begin{CD}
    S^{n}(S^1) @>>{\Psi(\mathbb{X}, \mathfrak{s}, f) }> S^n(S^1) \\
    @AA{}A  @AA{}A \\
    S^{n}(pt) @>>{\epsilon(\Phi(X, \mathfrak{s})\wedge \Phi(Y, \xi))}> S^n(pt)
    \end{CD}. 
    \]
    From the gluing formula \cite[Theorem 1.2]{iida2021seiberg} of the stable-cohomotopy refinement of the Kronheimer--Mrowka's invariant, we have 

    \[
        \Phi(X, \xi, \mathfrak{s})= \epsilon(\Phi(X, \mathfrak{s})\wedge \Phi(Y, \xi)) \in \pi_0^{st}(S^0). 
    \]
    
    From the definition of $\psi_{\xi, 0}(\mathbb{X}, \mathfrak{s}, f)$, we obtain the following formula by applying $KO^{n}$ to the above diagram:
    \begin{align*}
        \Phi(X,\xi ,\mathfrak{s})^* \tau_{n}= \psi_{\xi, 0}(\mathbb{X}, \mathfrak{s}, f), 
    \end{align*}
    where $\tau_n$ is the Thom class of $\tilde{KO}^{n}(S^n(pt))$. 
    Since the contact structure $\xi$ is induced by the symplectic structure $\omega$ on $X$, we can use \cite[Theorem 1.2.]{iida2021bauer} and we have that 
 $\Phi(X, \xi,\mathfrak{s})^* \tau_{n}$ is a generator. 
This finishes the proof. 
\end{proof}

\subsection{The choice of a stable framing of $H^+(\mathbb{X})$ and the family invariant}

In this subsection, we consider the effect of the choice of a stable framing $f$ of the real vector bundle $H^+(\mathbb{X}) \to S^1$ to the invariants $\psi_{-2}(\mathbb{X}', \mathfrak{s}', f)$ and $\psi_{\xi, -1}(\mathbb{X}, \mathfrak{s}, f)$. We assume the family $\mathbb{X}$ and the family of spin$^c$ structures $\mathfrak{s}$ satisfy the assumptions on \cref{sec_case1} or \cref{sec_case2}. 

There are two homotopy classes of stable framings on $S^1$ since $\pi^{st}_1(S^0) \cong \Omega^{fr}_1(pt) \cong \Z/2$. Let $f_0$ be a stable framing of $H^+(\mathbb{X})$ and let $f_1$ be the other stable framing. Then we have the following propositions:
\begin{proposition}\label{framing_dependence}
    Let $\psi_{-2}$ and $\psi_{\xi, -1}$ are the family invariant as we define in \cref{sec_case1} and \cref{sec_case2} respectively. 
    In case \eqref{closure} of \cref{main}, we assume that the Seiberg--Witten invariant of $(X', \mathfrak{s}')$ is odd. Then
    \begin{align}
        &\psi_{-2}(\mathbb{X}', \mathfrak{s}', f_0)+\psi_{-2}(\mathbb{X}', \mathfrak{s}', f_1)=1 \in \Z/2=\{0, 1\}, \label{formula1}\\
        &\psi_{\xi, -1}(\mathbb{X}, \mathfrak{s}, f_0)+\psi_{\xi, -1}(\mathbb{X}, \mathfrak{s}, f_1)=1 \in \Z/2=\{0, 1\}. \label{formula2}
    \end{align}
\end{proposition}
\begin{proof}
    The proof of \eqref{formula1} is similar to the proof of \eqref{formula2}. Here we prove \eqref{formula2}. 

    Let $E \to S^1$ be a oriented real vector bundle of rank $n$. Then the Thom class of $KO^{n}(E)$ depends on the choice of the spin structure of $E$. There are two spin structures on $E$ since $H^1(S^1, \Z/2) \cong \Z/2$. Let $\mathfrak{t}_E$ be one of the spin structure of $E$ and $\tau \in KO^{n}(E)$ be the corresponding Thom class. Then if we take the other spin structure on $E$, then the corresponding Thom class is $\tau \cdot l$ where $l$ is the M\"{o}bius line bundle. 
 
    Let $E_n$ be a real vector bundle $\underline{\R}^{\frac{\sigma(X)}{4}+\frac{1}{2}-d_3(-Y, [ \xi ])+n-1} \oplus H^+(\mathbb{X}) \to S^1$ where $n>>1$. Note that the rank of $E_n$ is $n$ since $-d_3(-Y, [\xi])=d_3(Y, [\xi])$ and $d_3(Y, [\xi])=-\frac{\sigma(X)}{4}-b^+(X)-\frac{1}{2}$ if $X$ is a compact almost complex bounding of $(Y, \xi)$. The framing $f$ of $H^+(\mathbb{X})$ defines a stable framing and a spin structure of $E_n$. Thus a representative of $\Psi_{\xi}(\mathbb{X}, \mathfrak{s}, f)$ is a proper map
    \[
        \Psi_{\xi}(\mathbb{X}, \mathfrak{s}, f)_n \colon S^1\times \R^n \to E_n \to S^1 \times \R^n.  
    \]
The last map depends on the choice of $f$. Let $\tau_n$ be a Thom class defined by the framing $f_0$. Then the Thom class defined by the framing $f_1$ is given by $\tau_n \cdot l$ where $l$ is the M\"{o}bius line bundle. Then 
\begin{align*}
    \Psi_{\xi}(\mathbb{X}, \mathfrak{s}, f_1)_n^*(\tau_n \cdot l)&=
    \Psi_{\xi}(\mathbb{X}, \mathfrak{s}, f_0)_n^*\tau_n \cdot l \\
    &=(\psi_{\xi, -1}(\mathbb{X},  \mathfrak{s}, f_0), \pm 1) \cdot l \in KO^0(S^1) 
\end{align*}
here we use the isomorphism 
\[
    KO^0(S^1) \cong \Z[l]/(l^2-1, 2l-2) \cong \Z/2\Z \cdot (l-1) \oplus \Z \cong KO^{-1}(pt) \oplus KO^0(pt). 
\]
Then we have 
\begin{align*}
(\psi_{\xi, -1}(\mathbb{X},  \mathfrak{s}, f_0), 1) \cdot l &=(\psi_{\xi, -1}(\mathbb{X},  \mathfrak{s}, f_0) + 1, 1).  
\end{align*}
This proves \eqref{formula2}. The proof of \eqref{formula1} is similar. 
\end{proof}

\section{Exotic diffeomorphisms}
In this section, we give a proof of \cref{main} and prove \cref{milnor_fiber}. 
\subsection{Proof of the main theorem}

\begin{proof}[Proof of \cref{main}.]

    Let us recall the assumptions in \cref{main}. 
    Let $X$ be a compact connected smooth spin $4$-manifold with boundary $Y$. We assume that $H_1(X, \Q)=0$ and $Y$ is a rational homology sphere. Let $f$ be a self-diffeomorphism of $X$ such that $f|_Y=id_Y$. We assume the following conditions:
\begin{itemize}
    \item $f^*=id$ on $H^*(X, \Z)$. 
    \item There is a self-diffeomorphism $\delta$ such that $\delta^2=f$ and $\delta|_Y=id_Y$.
    \item The dimension of the $-1$-eigenspace of $\delta^*$ in $H^+(X, \R)$ is even and not divisible by $4$.  
\end{itemize}
Furthermore, we assume one of the following two assumptions:

\begin{enumerate}
    \item There is a compact oriented smooth $4$-manifold $X^+$ such that $\partial X^+=-Y$ and $b^+(X \cup_Y X^+)>1$. Also we assume that there is a spin$^c$ structure $\mathfrak{s}$ on $X \cup_Y X^+$ such that $\mathfrak{s}|_X$ is a spin structure and the Seiberg--Witten invariant $SW(\mathfrak{s})$ is odd. 
    \item $X$ is symplectic and the canonical spin$^c$ structure induced by the symplectic structure coincides the spin$^c$ structure induced by a spin structure.  
\end{enumerate}
Let us assume the condition \eqref{symplectic}. The proof of the case \eqref{closure} is similar.  

We introduce three families: let $\mathbb{X}_0=S^1 \times X$ be a trivial family, $\mathbb{X}_{f}$ be the mapping torus of $f$, and $\mathbb{X}_{\delta}$ be the mapping torus of $\delta$. 
Let $\omega$ be the symplectic structure on $X$ and let us fix a spin structure $\mathfrak{s}$ that satisfies the condition \eqref{symplectic} i.e. $\mathfrak{s}$ is isomorphic to the canonical spin$^c$ structure $\mathfrak{s}_{\omega}$ as a spin$^c$ structure. 

We introduce family spin structures $\mathfrak{s}_0, \mathfrak{s}_f$, and $\mathfrak{s}_{\delta}$ on $\mathbb{X}_0, \mathbb{X}_f$, and $\mathbb{X}_{\delta}$ respectively by taking the mapping torus of $\mathfrak{s}$. Since all diffeomorphisms $id, f$ and $\delta$ are the identity on the boundary, there is a canonical lift of these maps to the spin structure. 

Since $f^*=id$ on $H^+(X, \R)$, then we can take a stable framing on $H^+(\mathbb{X}_0)$ and $H^+(\mathbb{X}_f)$ by fixing a basis of $H^+(X, \R)$. Let us call these framings $f_0$ and $f_{f}$ respectively. 

If $f$ is smoothly isotopic to the identity, we have 
\begin{align}\label{folse}
    \Psi_{\xi}(\mathbb{X}_0, \mathfrak{s}_0, f_0)=\Psi_{\xi}(\mathbb{X}_f, \mathfrak{s}_f, f_f) 
\end{align}
so that if we prove \cref{folse} is not true, then we proves the theorem. 

Let $\pi_2 \colon S^1 \to S^1$ be the double covering map. From the definition, we have 
\[
\pi_2^*\mathbb{X}_0 = \mathbb{X}_0,\; \pi_2^*\mathfrak{s}_0=\mathfrak{s}_0, \; \pi_2^*\mathbb{X}_\delta = \pi_2^*\mathbb{X}_f, \; \pi_2^* \mathfrak{s}_{\delta}=\mathfrak{s}_f.  
\]
Let $f_\delta$ be a framing of $H^+(\mathbb{X}_\delta)$. Then from the construction of the finite dimensional approximation, we have 
\[
\pi_2^* \Psi_{\xi}(\mathbb{X}_0, \mathfrak{s}_0, f_0)= \Psi_{\xi}(\mathbb{X}_0, \mathfrak{s}_0, f_0), \; \pi_2^*\Psi_{\xi}(\mathbb{X}_\delta, \mathfrak{s}_\delta, f_{\delta})=\Psi_{\xi}(\mathbb{X}_f, \mathfrak{s}_f, \pi_2^*f_\delta). 
\]
Note that $\pi_2^*f_{\delta}$ is a framing on $\pi_2^* H^+(\mathbb{X}_{\delta}) = H^+(\mathbb{X}_f)$. Therefore $\Psi_{\xi}(\mathbb{X}_0, \mathfrak{s}_0, f_0)^*\tau$ and $\Psi_{\xi}(\mathbb{X}_f, \mathfrak{s}_f, \pi^*f_{\delta})^*\tau$ are in the image of 
\[
\pi_2^* \colon KO^0(S^1) \to KO^0(S^1) 
\]
where $\tau$ is a Thom class of $KO^0(S^0) \cong \tilde{KO}^n(S^n(S^1))$. 
If  $l$ is a M\"{o}bius bundle, one can see that $\pi_2^*l=1$. Then $\pi^*_2$ is given by the map 
\[
\Z[l]/(2(l-1), (l-1)^2) \cong \Z/2 \oplus \Z \ni (a, b) \mapsto (0, b). 
\]
Thus we have $\psi_{\xi, -1}(\mathbb{X}_0, \mathfrak{s}_0, f_0)$ and $\psi_{\xi, -1}(\mathbb{X}_f, \mathfrak{s}_f, \pi_2^*f_\delta)$ is $0$ for any $f_{\delta}$. 

We only have to prove that 
\begin{align}\label{conclusion}
    \psi_{\xi, -1}(\mathbb{X}_f, \mathfrak{s}_f, \pi_2^*f_\delta) \neq \psi_{\xi, -1}(\mathbb{X}_f, \mathfrak{s}_f, f_f)
\end{align}
for some $f_{\delta}$. 
From the assumption on $\delta$, there is a natural isomorphism $H^+(\mathbb{X}_\delta) \cong \underline{\R}^{b^+(X)-4k-2} \oplus l^{4k+2}$ for some $k \ge 0$. Let $f_{2l}$ be a framing on $l \oplus l$. 
One can see that there is a natural isomorphism $\pi_2^* (l\oplus l) \cong S^1 \times \R^2$ since $l=\R \times [0, 1]/(x, 0) \sim (-x, 1)$. 
Moreover, the framing $\pi^*_2 f_{2l}$ of $\underline{R}^2 = S^1 \times \R^2$ defines a non-trivial element of $\pi_1(SO(2)) \cong \Z$ and its $\bmod 2$ reduction is also non-trivial. 
Let $f_{\delta}$ be a framing given by $f_{\text{standard}}^{b^+(X)-4k-2} \oplus f_{2l}^{2k+1}$ where $f_{\text{standard}}$ be the framing of $S^1 \times \R$ that is given by the trivialization. Then we see that $\pi^*_2 f_{\delta}$ is not homotopic to the framing $f_0$ since $f_0=f_{\text{standard}}^{b^+(X)}$. Therefore, from \cref{framing_dependence}, we have \eqref{conclusion}. 
\end{proof}

\subsection{Examples}\label{examples}
Let us prove \cref{milnor_fiber}. Let us recall the boundary Dehn twist on $M_c(p, q, r)$. 
The boundary of $M_c(p, q, r)$ is the Brieskorn sphere 
\[
\Sigma(p, q, r):= \{ (x, y, z) \in \C^3 \mid x^p + y^q + z^r=0, \lvert x \rvert^2 + \lvert y \rvert^2 + \lvert z \rvert^2=1 \}.
\]

We can define a $U(1)$-action on $\Sigma(p, q, r)$ by $u \cdot (x, y, z) = (u^{qr}x, u^{rp}y, u^{pq}z)$ for $u \in U(1)$. 
Let $[0, 1] \times \Sigma(p, q, r)\cong  C \subset M_c(p, q, r)$ be a collar neighborhood of the boundary $\partial C \cong \{0\} \times \Sigma(p, q, r)$. We identify $C$ as the cylinder $[0, 1] \times \Sigma(p, q, r)$.  We define
\[
f' \colon C \to C ,\;(t, (x, y, z)) \mapsto (t, e^{ i \rho(t)}(x, y, z))
\]
here $\rho \colon [0, 1] \to [0, 2\pi]$ be a smooth, non-decreasing  function with $\rho(t)=0$ if $0\le t \le 1/3$ and $\rho(t)=2\pi$ if $2/3 \le t \le 1$. Since $f'|_{\partial C}=id$, we can extend $f'$ to $M_c(p, q, r)$ by gluing to the identity on $M_c(p, q, r)\setminus C$. This is the definition of the diffeomorphism $f$.

\begin{lemma}\label{2qr}
    Let $\iota \colon M_c(2, q, r) \to M_c(2, q, r)$ is an involution that is given by $(x, y, z) \mapsto (-x, y, z)$. Then if $q, r$ are odd, coprime, and 
    \[
    \frac{(q-1)(r-1)}{4}
    \] is an odd number, there is a self diffeomorphism $\delta$ on $M_c(2, q, r)$ that satisfies the following properties:
    \begin{itemize}
        \item $\delta$ is the identity on the boundary.
        \item Let $f$ be the boundary Dehn twist. Then we have $\delta^2=f$.  
        \item The induced map $\delta^* \colon H^+(M_c(2, q, r), \R) \to H^+(M_c(2, q, r), \R)$ is order $2$ and the dimension of the eigenspace of eigenvalue $-1$ is even and not divisible by $4$.  
        \item $\delta$ commutes with the involution $\iota$. 
    \end{itemize}
\end{lemma}
\begin{proof}
    Let $\delta'$ be a diffeomorphism of the collar neighborhood $C \subset M_c(2, q, r)$ by 
    \[
        \delta' \colon C \to C \; ; \; (t, (x, y, z)) \mapsto (t, e^{i \frac{\rho(t)}{2}}(x, y, z)). 
    \]
    Then $\delta'$ is the identity on the boundary $\partial C \cong \{0\} \times \Sigma(2, q, r)$ and coincides with the involution $\iota$ on $\{1\} \times \Sigma(2, q, r)$. 
    Thus we can extend $\delta'$ to the diffeomorphism of $M_c(2, q, r)$ by gluing with $\iota$ on $\{1\} \times \Sigma(2, q, r)$. Let us call $\delta$ this extended diffeomorphism. 
    Since $\delta'$ commutes with $\iota$, one can see that $\delta$ commutes with $\iota$. From the definition of $f$ and $\delta$, $\delta^2=f$. 

    The induced map $\delta^* \colon H^+(M_c(2, q, r), \R) \to H^+(M_c(2, q, r), \R)$ is order $2$ because $\delta$ is homotopic to the involution $\iota$. 
    The induced map $\iota^*$ is $-1$ on $H^+(M_c(2, q, r), \R)$ since $\iota$ is the covering involution of the double branched cover of $D^4$ along a surface $S$ in $D^4$ that bounds $(q, r)$-torus with minimal $4$-genus.  
    Thus we only have to prove that the dimension of $H^+(M_c(2, q, r), \R)$ is even and not divisible by $4$ if the condition of $(q, r)$ is satisfied. From the standard argument of the algebraic topology, we see $b_2(M_c(2, q, r))=g(S)$ where $g(S)$ is the genus of $S$. Then $b^+(M_c(2, q, r))$ is given by 
    \[
    g(S)+\frac{\sigma(T(q, r))}{2} 
    \]
    where $\sigma(T(q, r))$ is the knot signature of $(q, r)$-torus knot. 
    From the Milnor conjecture, the $4$-genus of $T(q, r)$ is $\frac{(q-1)(r-1)}{2}$. From the result proved by Viro \cite{viro1973branched} implies that $\sigma(T(q, r))$ coincides with a value $t(q, r, 2)$ which was computed by Hirzebruch--Zagier \cite{hirzebruch1974atiyah}.  From \cite[Corollary,p.122]{hirzebruch1974atiyah}, we have $t(q, r, 2)$ is divisible by $8$ so that this proves the lemma.     
\end{proof}

Let us prove \cref{milnor_fiber}. 
\begin{proof}[Proof of \cref{milnor_fiber}]
    We want to apply \cref{main} to the boundary Dehn twist on $M_c(2, q, r)$ where $q, r$ are odd and coprime positive integer bigger than $2$. 
    Let us consider the embedding $i \colon M(p, q, r) \to \C^3$ in general. Since the  normal bundle of $M(p, q, r)$ is trivial as a complex line bundle so that the canonical bundle $K$ on $M(p, q, r)$ coincides with the pullback of the canonical bundle of $\C^3$ and this is trivial. 
    Thus we have that the canonical bundle $K$ on $M(p, q, r)$ is trivial. This implies that $M_c(p,q,r )$ is spin and the spin$^c$ structure induced by the K{\"a}hler form coincides with the spin structure. 
    Then we see that $M_c(p, q, r)$ satisfies the second condition \eqref{symplectic}. Moreover, $M(p, q, r)$ is simply connected since it is a affine algebraic surface and if $(p, q, r)$ are pairwise comprime, then $\Sigma(p, q, r)$ is homology sphere. 
    
    From \cref{2qr}, if $p=2$ and $q, r$ are odd, coprime integer bigger than $2$, one can see that $f$ and $\delta$ satisfies the conditions in \cref{main} therefore $f$ is not smoothly isotopic to the identity. 

    Since $M_c(2, q, r)$ is simply connected and $f^*$ acts trivially on the cohomology group of $M_c(2, q, r)$ so that \cite[Theorem A]{orson2022mapping} implies $f$ is topologically isotopic to the identity relative to the boundary. This proves the theorem. 
\end{proof}

\begin{proof}[Proof of \cref{milnor2}]
    Let us recall that if two triples $(a, b, c)$ and $(a', b', c')$ satisfy $a \le a'$, $b \le b'$, $c \le c'$, and $2 \le a ,b ,c$, then there is an embedding 
    \begin{align}\label{embedding_abc}
    i \colon M_c(a,b, c) \to M_c(a', b', c') . 
\end{align}
       A proof is given in the argument in \cite[pp264-265]{gompf20234}. 
    If $(p, q, r)$ satisfies the assumptions in the theorem, them we have an embedding $M_c(2, 3, 7) \to M_c(p, q, r)$. 
    Let $f'$ be the boundary Dehn twist on $M_c(p, q, r)$. 
    Since $f'$ is identity on the boundary, we can extend $f'$ to $M_c(p, q, r)$ by gluing with the identity on $M_c(p, q, r) \setminus M_c(2, 3, 7)$. 
    We write $f$ the extended diffeomorphism. 
    The induced map $f^*$ on the cohomology on $M_c(p, q, r)$ is the identity since $f$ is topologically isotopic to the identity. What we only have to do is to find a diffeomorphism $\delta$ that satisfies the assumptions in \cref{main} because one can see that other assumptions are satisfied from the argument in the Proof of \cref{milnor_fiber}. 

    Let $\delta$ be the diffeomorphism of $M_c(2, 3, 7)$ that is given in the proof of \cref{2qr}. 
    Since $\delta$ is the identity on the boundary, we extend $\delta$ to the diffeomorphism of $M_c(p, q, r)$ by gluing with the identity on $M_c(p, q, r) \setminus M_c(2, 3, 7)$. 
    One can see that $\delta^2=f$ on $M_c(p, q, r)$ so that the induced map $\delta^*$ on $H^+(M_c(p, q, r), \R)$ is an involution. Let us show that the $-1$ eigenspace of $\delta^*$ is isomorphic to $H^+(M_c(2, 3, 7), \R)$. 
    
    Firstly, let us prove that the map 
    \[
   i^* \colon H^2(M_c(a', b', c'), \R) \to H^2(M_c(a, b, c), \R)
    \]
    induced by the embedding \eqref{embedding_abc} is surjective and if the support of a diffeomorphism $\delta$ is contained in $M_c(a, b, c)$ then $\delta^*$ acts trivially on the kernel of $i^*$. 
    It is enough to show above in the case that $a=a', b=b'$ and $c <c'$. From the argument in \cite[pp264-265]{gompf20234}, $M_c(a, b, c')$ is given by the fiber-sum of $M_c(a, b, c)$ and $M_c(a, b, c'-c)$. The intersection $A:=M_c(a, b, c) \cap M_c(a, b, c'-c) \subset M_c(a, b, c')$ is homeomorphic to 
    \[
    [0, 1] \times \{ (x, y) \in \C^2 \mid x^a +y^b =\epsilon , \; \lvert x \rvert^2 + \lvert y \rvert^2\le1\}
    \]
    for some small $\epsilon \in \C\setminus \{0\}$. Therefore $H^2(A, \R)=0$. Since $M(a, b, c)$ and $M(a, b, c'-c)$ are simply connected, the Mayer--Vietoris exact sequence induce the following short exact sequence:
    \[
    0 \to H^1(A, \R) \to H^2(M_c(a, b, c'), \R) \to H^2(M_c(a, b, c), \R) \oplus H^2(M_c(a, b, c'-c), \R) \to 0. 
    \]
    Therefore the map $ i^* \colon H^2(M_c(a, b, c'), \R) \to H^2(M_c(a, b, c), \R)$ is surjective and the above sequence is funtorial so that $\delta^*$ is trivial on the kernal of $i^*$. 

    From the above argument, if $(p, q, r)$ satisfies the assumptions of the statement of \cref{milnor2}, then 
    \[
    H^2(M_c(p, q, r), \R) \to H^2(M_c(2, 3, 7), \R)
    \]
    is surjective and the $-1$ -eigenspace of $\delta^*$ is isomorphic to  $H^2(M_c(2, 3, 7), \R) \cong H^2(M_c(p, q, r),\partial M_c(p, q, r), \R)$. 
    If $\xi , \eta \in H^2(M_c(p, q, r),\partial M_c(p, q, r), \R)$, $\delta^* \xi = -\xi$, $\delta^* \eta=\eta$, then $\delta^* (\xi \wedge \eta)=\delta^*\xi \wedge \delta^*\eta=-\xi \wedge \eta$ but $\delta^*$ is the identity on $H^4(M_c(p, q, r), \partial M_c(p, q, r),\R)$. Thus $\xi \wedge \eta=0$. This argument implies the intersection form $\sigma(M_c(p, q, r))$ has the following splitting:
    \[
        \sigma(M_c(p, q, r)) \cong \sigma(M_c(2, 3, 7)) \oplus \sigma(M_c(p, q, r))|_{\ker(H^2(M_c(p, q, r), \R) \to H^2(M_c(2, 3, 7), \R))}. 
    \]
    Therefore the $-1$-eigenspace of $\delta^*$ in $H^+(M_c(p, q, r), \R)$ is isomorphic to $H^+(M_c(2, 3, 7), \R)$. 
    Since $\dim_{\R}(H^+(M_c(2, 3, 7), \R)) = 2$, this proves \cref{milnor2}. 
\end{proof}

\bibliographystyle{jplain}
\bibliography{bibidatabase}

\end{document}